\documentclass[pdflatex,sn-mathphys-ay, Numbered]{sn-jnl}

\usepackage{stmaryrd}
\usepackage{bbold}
\usepackage{graphicx}%
\usepackage{multirow}%
\usepackage{amsmath,amssymb,amsfonts}%
\usepackage{amsthm}%
\usepackage{mathrsfs}%
\usepackage[title]{appendix}%
\usepackage{xcolor}%
\usepackage{textcomp}%
\usepackage{manyfoot}%
\usepackage{booktabs}%
\usepackage{algorithm}%
\usepackage{algorithmicx}%
\usepackage{algpseudocode}%
\usepackage{listings}%
\usepackage{orcidlink}

\theoremstyle{thmstyleone}%
\newtheorem{theorem}{Theorem}
\newtheorem{theo}{Theorem}[section]
\newtheorem{propo}{Proposition}[section]
\newtheorem{lemma}{Lemma}[section]

\newtheorem{exam}{Example}[section]

\newcommand{\kk}{\mathbb{k}}
\newcommand{\Ck}{\mathbb{n_\kk}}
\newcommand{\Do}{\mathrm{D}_\kk}
\newcommand{\Ha}{\mathrm{H}_\kk}

\newcommand{\Card}{\mathrm{Card}}

\newcommand{\NN}{\mathbb{N}}
\newcommand{\RR}{\mathbb{R}}

\DeclareMathOperator{\supp}{supp}

\begin{document}

\title[Stochastic Aggregation-Diffusion Equation: Analysis via Dirichlet Forms]{Stochastic Aggregation-Diffusion Equation: Analysis via Dirichlet Forms}

\author[1]{\fnm{Jaouad} \sur{Bourabiaa}}\email{jaouad.bourabiaa@gmail.com}
\equalcont{These authors contributed equally to this work.}

\author*[1]{\fnm{Youssef} \sur{ Elmadani}\orcidlink{0000-0001-7970-7575}}\email{ elmadanima@gmail.com}

\author[1]{\fnm{Abdelouahab} \sur{Hanine}}\email{abhanine@gmail.com}
\equalcont{These authors contributed equally to this work.}

\affil[1]{\orgdiv{Laboratoire d'Analyse Mathématique et Applications}, \orgname{Mohammed V University}, \orgaddress{\street{Faculty of Sciences}, \city{Rabat}, \postcode{B.P. 1014 Rabat}, \country{Morocco}}}


\abstract{In this article, we study the stochastic aggregation-diffusion equation with a singular drift represented by a monotone radial kernel. We demonstrate the existence and uniqueness of a diffusion process that acts as a weak solution to our equation. This process can be described as a distorted Brownian motion originating from a delocalized point. Utilizing Dirichlet form theory, we prove the existence of a weak solution for a quasi-everywhere point in a state space. However uniqueness is not assured for solutions commencing from points outside polar sets, and explicitly characterizing these sets poses a significant challenge. To address this, we employ the \(\mathrm{H_2}\)-condition introduced by  \cite{albeverio2003strong}. This condition provides a more thorough understanding of the uniqueness issue within the framework of Dirichlet forms. Consequently the \(\mathrm{H_2}\)-condition is pivotal in enhancing the analysis of weak solutions, ensuring a more detailed comprehension of the problem. An explicit expression for the generalized Schrödinger operator associated with certain kernels is also provided.}




\maketitle

\section{Introduction}\label{sec1}
The Keller-Segel system models the phenomenon of chemotaxis, describing the movement of cells toward chemicals. For a comprehensive overview of chemotaxis, refer to foundational works such as \cite{maini2001applications,book:114270}, and for more detailed discussions and citations can be found in articles such as \cite{Arumugam2020KellerSegelCM, HorstmannFU}. The evolution of cell density $\rho_t$ is determined by the nonlinear aggregation–diffusion equation
\begin{equation}\label{1}
	\begin{cases}
  \partial_t \rho_t = \Delta \rho_t -\chi \nabla \cdot (\rho_t (\nabla \mathbb{k}\ast \rho_t)(x)) , \quad \quad  x \in \mathbb{R}^d, t > 0,\\
		\rho( x, 0)= \rho_0( x) \ge 0 .
	\end{cases}
\end{equation} 
Cells spread throughout the growth zone in search of nutrients, while some begin to produce a chemical signal with a force field kernel $\mathbb{k}$. The sensitivity, denoted by the parameter $\chi > 0$, determines the extent of their response to the chemoattractant and reflects the system's nonlinearity. Accordingly the nature of the interaction kernel can lead to diverse phenomena, ranging from the self-organization of chemotactic movement \cite{Blanchet2007InfiniteTA,Keller1970InitiationOS,book:58110}, biological swarming \cite{Burger2013StationarySA,Topaz2005ANC}, and cancer invasion \cite{Domschke2014MathematicalMO,Chaplainb2008MathematicalMO}.\\

\cite{blanchet2006two} identified a critical mass value \(M_c = 8\pi/\chi\) in two dimensions for the logarithmic kernel. Solutions with mass below this threshold exhibit self-similar behavior and exist globally, while those above it blow up in finite time. For the critical mass, solutions can both exist globally and form a Dirac mass at the center as \(t \rightarrow \infty\) (see \cite{Blanchet2007InfiniteTA}). However changes in the kernel affect how the critical mass is determined, necessitating a separate study for each type of kernel see for instance \cite{Carrillo2018AggregationDiffusionED}.\\

 We can ensure that \(\rho_0\)  remains a probability density function. Thus  equation \eqref{1} corresponds to the Kolmogorov forward (or Fokker-Planck) equation associated with the stochastic differential equation
\begin{equation}\label{2}
	\begin{cases}
		dX_t = \sqrt{2}dW_t -\chi (\nabla\mathbb{k}\ast \rho_t)(X_t) dt,\\
		\rho_t(x) dx =\mathcal{L}(X_t),
	\end{cases}
\end{equation}
where $(W_t)_{t\geq 0}$ is the Brownian motion in $\RR^d$. This derivation aims to clarify how deterministic global behavior develops from stochastic particle interactions, resulting in a simplified macroscopic description. This concept is applied in some mathematical models (see, e.g., \cite{Haskovec2009StochasticPA}). By employing the propagation of chaos approach, we obtain equation \eqref{2} as the limit as $N \to \infty$ of a linear system of stochastic differential equations with mean-field interactions given for $i = 1,\ldots, N$
\begin{equation}\label{3}
	\begin{cases}
		dX_t^{i} = \sqrt{2}dW_t^{i} -\frac{\chi}{N}\sum_{j=1, i \neq j}^N \nabla\mathbb{k}(X_t^{i}-X_t^{j})dt,\\
		X_0=x \sim \otimes_{i=1}^N \rho_0.
	\end{cases}
\end{equation}
Here $W^1_t,W^2_t,\ldots,W^N_t$ are independent Brownian motions in $\mathbb{R}^d$. We denote $\mathcal{W}_t:=(W^1_t,W^2_t,\ldots,W^N_t)$ and $\mathcal{X}_t:=(X^1_t,X^2_t,\ldots,X^N_t)$  as a solution of equation \eqref{3}.\\

For the logarithmic kernel in two dimensions (\(d=2\)), \cite{Cattiaux2016The2S} studied the equation \eqref{3} using the Dirichlet forms approach. Related works include \cite{cattiaux2024entropy, Fournier2015StochasticPA, Fournier2021CollisionsOT}.

\section{Results}\label{sec2}
To state our main result, we introduce some notations. Throughout the paper, $\mathbb{k}$ denotes a radial kernel,  specifically $\mathbb{k}(x)=\mathbb{k}(\|x\|)$ for $x\in \mathbb{R}^d$, satisfying the condition
$$[\mathrm{H}]_p: \qquad\int_0 |\mathbb{k}'(t)|^{p}\exp\left(-\frac{\chi}{N}\mathbb{k}(t)\right) t^{d-1}dt<\infty, \quad N\geq 2,\,\,\chi>0,\,\,d\geq 2,\,\,p>1.$$
Note that the class $[\mathrm{H}]_p$ differs from the Kato class $\mathrm{K}_d$ in the case of radial kernels; e.g., for $\mathbb{k}_\alpha(x)=|x|^{-2}[-\log|x|]^{-\alpha}$, if $d\geq 3$, then $\mathbb{k}_\alpha$ is in $\mathrm{K}_d$ only if $\alpha>1$, whereas $\mathbb{k}_\alpha$ belongs to $[\mathrm{H}]_p$ for all $\alpha$ and all $p$. For further details on the Kato class see \cite{simon1982schrodinger}.\\

Let $q\in \mathbb{N}$, and define $E_{q}$ as the set given by 
\begin{equation*}
	E_{q} = \left\{ x \in (\mathbb{R}^d)^N \mid \forall \mathrm{K} \subseteq \llbracket 1~, ~N\rrbracket \text{ with } |\mathrm{K}| = q, \,\text{and}\, \sum_{i,j \in \mathrm{K}} \|x^i - x^j\|^2 > 0  \right\},
\end{equation*}
where $\llbracket 1~,~ N\rrbracket$ denotes the set of integers from $1$ to $N$. Set $E_q$ ensures the absence of clusters containing $q$ or more particles occupying the same position. Let $\Ck$ now be defined as 
$$\Ck=\max\left\{n\in \llbracket 2~;~ N+1 \rrbracket: \,\, \int_0\exp\left(-\frac{n\chi}{d N}\,\kk(r)\right)dr<\infty \right\}.$$
The main result of this paper is the following\\

\begin{theorem}\label{1.1}
	Let $N\geq 2$, $\chi>0$ and $d\geq 2$. Assume that $\kk$ is a radial kernel such that $|\kk(0^+)|=+\infty$. 
	Suppose that $[\mathrm{H}]_p$ holds for some $p>d N$. If $\kk$ is an increasing (resp., decreasing) function, we let $\mathbb{E}_{\Delta}= E_{\Ck} \cup \{\Delta\}$ (resp., $\mathbb{E}_{\Delta}=E_2\cup \{\Delta\}$), where  $\Delta$ is a cemetery point, then there exists a unique weak solution $\mathbb{X}=((\mathcal{X}_t)_{t \geq 0}, (\mathbb{P}_x)_{x \in \mathbb{E}_{\Delta}})$ of the stochastic differential equation \eqref{3} such that $\mathbb{P}_x$-almost surely for any $x \in \mathbb{E}_{\Delta}\setminus \{\Delta\}$,
	\begin{equation}\label{5}
		\mathcal{X}_t = x + \sqrt{2} \mathcal{W}_t + \int_0^t\frac{\nabla m}{m}(\mathcal{X}_s) \, ds, \quad t \geq 0,
	\end{equation}
	where $$m(x) = \exp\left(-\frac{\chi}{N} \sum_{1\leq i \neq j \leq N} \kk(\|x^i-x^j\|)\right),\quad x\in (\RR^d)^N .$$
	$\mathbb{X}$ is called distorted Brownian motion.
\end{theorem}
\section{Dirichlet forms}\label{Dirichletform}
In this section, we provide a summary of Dirichlet form theory and its connections with stochastic process theory. For a more detailed discussion, the reader can refer to \cite{book:597396, ma2012introduction}. \\

Let \(E \) be a locally compact, separable metric space with Borel $\sigma$-algebra \(\mathcal{B}(E)\), and let \(\mu\) be a positive Radon measure with full support on  \(E\). The scalar product in \(L^2(E, \mu)\) is denoted by \(\langle \cdot, \cdot \rangle\), while  \(C_0(E)\) is the space of continuous functions with compact support on \(E \).

\subsection{Basic definitions} A symmetric form on $L^2(E,\mu)$ is a positive bilinear form $\mathcal{E}$ defined on a dense linear subspace $\mathcal{D(E)}\subset L^2(E, \mu)$. The space $\mathcal{D(E)}$ becomes a pre-Hilbert space with respect to the inner product $\mathcal{E}_1(\cdot, \cdot) = \mathcal{E}(\cdot, \cdot) + \langle \cdot, \cdot \rangle$. We say that the symmetric form $(\mathcal{E}, \mathcal{D(E)})$  is closed if $\mathcal{D(E)}$ is complete. A Dirichlet form on $L^2(E, \mu)$ is a closed symmetric form $(\mathcal{E}, \mathcal{D(E)})$ that is also Markovian, i.e., for any contraction $\mathrm{T}$ on $\mathbb{R}$ and any $u \in \mathcal{D(E)}$,  $\mathrm{T}(u) \in \mathcal{D(E)}$ and  $\mathcal{E}(\mathrm{T}(u), \mathrm{T}(u)) \leq \mathcal{E}(u, u)$. If \(\mathcal{E}\) admits a closed extension, it is said to be closable.  However this extension is not unique. To ensure the form remains Markovian, we work with the smallest closed extension that preserves the Markovian property see for instance \cite[Theorem 3.1.1]{book:597396}. \\

A Dirichlet form  \(\mathcal{E}\) on $L^2(E,\mu)$  is said to be regular if and only if the set \(\mathcal{C}:=\mathcal{D(E)} \cap C_0(E)\) is dense in both  \(\mathcal{D(E)}\) with respect to the \(\mathcal{E}_1\)-norm and \(C_0(E)\) with respect to the uniform norm.  For regular forms, Beurling-Deny and LeJan provide a general integral representation
\[
\mathcal{E}(u, v) = \mathcal{E}^{(c)}(u, v) + \int_{E \times E \setminus \mathrm{d}} (u(x) - u(y))(v(x) - v(y)) J(dx, dy) + \int_E u(x) v(x) k(dx),
\]
where
\begin{itemize}
	\item \(\mathcal{E}^{(c)}\) is a symmetric form with domain \(\mathcal{C}\) and satisfies the strong local property, that is,  \(\mathcal{E}(u, v) = 0\) for \(u, v \in \mathcal{C}\) such that  \(v\) is constant on a  neighborhood of \(\text{supp}[u]\).
	\item $J$ is a symmetric positive Radon measure called the \textit{jumping measure}.
	\item $k$ is a positive Radon measure called the \textit{killing measure}.
\end{itemize}
There exists a one-to-one correspondence between the family of closed symmetric forms $\mathcal{E}$ and the family of negative self-adjoint operators $\mathcal{L}$ on $L^2(E,\mu)$. This correspondence satisfies
$$\begin{cases}
	\mathcal{E}(u,v)=\langle -\mathcal{L} u,v\rangle, \quad u \in \mathcal{D}(\mathcal{L}), v \in \mathcal{D(E)},\\
	\mathcal{D(L)} \subset \mathcal{D(E)}.
\end{cases}$$
Furthermore, let \(\mathrm{(T_t)_{t\geq 0}}\) be the semigroup associated with the generator $\mathcal{L}$. The Markovian nature of the form is reflected in the Markovian property of the semigroup $\mathrm{T_t}$, i.e., if $0 \leq u \leq 1$ almost everywhere (a.e.) with respect to $\mu$, then $0 \leq \mathrm{T_t} u \leq 1$ a.e. for all $t\geq 0$. \\

\subsection{Symmetric Hunt process}
We introduce a new point $\Delta \notin E$. Let $E_{\Delta} = E \cup \{\Delta\}$ denote the one-point compactification of $E$ and let $\mathcal{B}(E_{\Delta}) = \sigma(\mathcal{B}(E), \{\Delta\})$ be the sigma algebra generated by $\mathcal{B}(E)$ and $\{\Delta\}$. A Hunt process $\mathbb{X} = (\Omega, \mathcal{F}, (\mathcal{F}_t)_{t \geq 0}, \tau, (X_t)_{t \geq 0}, (\mathbb{P}_x)_{x \in E_{\Delta}})$ is a homogeneous, quasi-left-continuous Markov process with càdlàg paths (i.e., right-continuous with left-limits) and lifetime $\tau = \inf \{ t \geq 0: X_t = \Delta \}$, which satisfies the strong Markov property. See for instance Appendix A.2 of \cite{book:597396}. For any $x \in E$ and $A \in \mathcal{B}(E)$, we define
\[
\mathbb{P}_x(X_t \in A) = \mathrm{P_t} \mathbb{1}_A(x),
\]
where $\mathrm{(P_t)_{t \geq 0}}$ is a Markovian transition function (or semigroup) associated with the $\mu$-symmetric Hunt process $(X_t)_{t \geq 0}$, i.e., $\int_E u \mathrm{P_t} v \, \mathrm{d} \mu = \int_E v \mathrm{P_t} u \, \mathrm{d} \mu$ for all bounded measurable functions $u, v$. Therefore constructing a regular Dirichlet form from a Hunt process is not a difficult task. \\

The converse turns out to be true as well, but the problem lies in the construction of the transition function associated with this process. Let $\mathcal{E}$ be a regular Dirichlet form on $L^2(E,\mu)$. We can associate a Markovian semigroup $\mathrm{(T_t)_{t\geq 0}}$ with $\mathcal{E}$. A natural approach to construct transition probabilities is to define \( \mathrm{P_t}(x,A)= \mathrm{T_t}\mathbb{1}_A(x) \) for a.e \( x \in X \) and \( A \in \mathcal{B}(E) \). However the Chapman-Kolmogorov equations, which are essential for consistency in a Markov process, only hold for sets of non-zero capacity with respect to $\mathcal{E}$. Therefore constructing a Hunt process $\mathbb{X}=(X_t)_{t\geq 0}$ associated with $\mathcal{E}$ requires systematically avoiding sets of zero capacity. If, moreover \(\mathcal{E}\) possesses the strong local property, \(\mathbb{P}_x\) can be modified only on a properly exceptional set such that \(\mathbb{P}_x\left(\tau < \infty, X_{\tau -} = \triangle\right) = 1\) for all \(x \in E\). This ensures that for each \(x \in E\), \(\mathbb{P}_x\)-almost surely, the path \(t \mapsto X_t\) is continuous from \([0, \infty)\) to \(E_{\triangle}\). Consequently \(\mathbb{X}\) is a diffusion process. For a detailed exploration of the subject, we refer to \cite[Chapter 7]{book:597396}.

\subsection{Distorted Brownian Motion}\label{sec.3} Let $E$ be a locally compact subset of $\RR^d$, where $d\geq 2$, and let $\mu(dx)=m(x) dx$. Define the set
$$
S(m):=\left\{x\in E\,\,:\,\, \int_{y\in \mathrm{B}(x,\epsilon)}\frac{1}{m(y)}dy=\infty\,\, \text{for all}\,\,  \epsilon>0\right\}.
$$
The bilinear form \(\mathcal{E}_m\) associated with \(m\) is defined by
$$
\mathcal{E}_m(u, v) = \frac{1}{2}\int_{\mathbb{R}^d} \nabla u \cdot \nabla v \, m(x) \, dx, \quad u, v \in C_0^\infty(E).
$$
According to  \cite{rockner1985dirichlet}, a sufficient condition for \((\mathcal{E}_m,C_0^\infty(E))\) to be closable is that the Lebesgue measure of \( S(m) \) is zero. Let $(\mathcal{E}_m,\mathcal{D}(\mathcal{E}_m))$  be the smallest extension such that $\mathcal{E}_m$ is a Dirichlet form. The associated infinitesimal generator $\mathcal{L}$ is then given by 
$$\mathcal{L}:= \frac{1}{2}\Delta+\frac{\nabla m}{2m}(x).\nabla.$$ 
The distorted Brownian motion (DBM)  associated with the
measure $\mu$ is  a diffusion process  $\mathbb{X} = ((X_t)_{t \geq 0}, (\mathbb{P}_x)_{x \in E_\Delta})$
with infinitesimal generator $\mathcal{L}$. Fukushima has explained, in a broader context, the identification of the DBM as a weak solution of the stochastic differential equation
\begin{equation}\label{DBM}
	X_t = x + \sqrt{2} W_t + \int_0^t \frac{\nabla m}{m}(X_s) \, ds, \quad t \geq 0
\end{equation}
for quasi-every starting point $x$. This approach causes us to lose a lot in terms of the uniqueness of the solution. However  under condition $\frac{\|\nabla m\|}{m}  \in L^{d+\epsilon}_{\text{loc}}(\mathbb R^d, \mu)$ for some $\epsilon > 0$, a weak solution, in the sense of the martingale
problem,  has been constructed starting from any point $x \in \{m \neq 0\}$  \cite{albeverio2003strong}. \cite{Shin2014OnTS} explain the connection between the weak solution of equation \eqref{DBM} and the generator $\mathcal{L}$. They develop general tools to apply Fukushima's absolute continuity condition, enabling the construction of a Hunt process with a transition function that admits a density $p_t(x,y)$ with respect to $\mu$.  Solutions can start from any point in the explicitly specified state space.
\section{Proof of Theorem \ref{1.1}}
In this section, we construct a unique weak solution for the system (\ref{3}). We begin with results that extend those in \cite{Fournier2021CollisionsOT} to general monotone radial kernels $\kk$. \\

Let  $x=(x^1,x^2,\ldots,x^N)\in (\RR^d)^N$ be a point in the $N$-fold product space of $\RR^d$ and let $\epsilon>0$. We define
$$\mu(dx)=m(x)dx, \quad \text{where} \quad m(x) = \exp\left(-\frac{\chi}{N} \sum_{(i,j)\in \llbracket 1~;~ N\rrbracket^2\setminus d} \kk(\|x^i-x^j\|)\right),\nonumber$$
and consider 
\begin{itemize}
	\item $\mathrm{B}(x,\epsilon):=\{y\in (\RR^d)^N:\,\, |x-y|<\epsilon\}$, the open ball centered at $x$ with radius $\epsilon$ in the product space, where $|x-y|^2=\sum_{i=1}^{N}\|x^i-y^i\|^2$.\\
	\item  $\mathrm{b}(x^i,\epsilon)=\{y^i\in \RR^d:\,\, \|x^i-y^i\|<\epsilon\}$, the open ball centered at $x^i$ with radius $\epsilon$ in $\RR^d$ (individual component space). 
\end{itemize}
It is well known that
$$
\mathrm{B}(x,\epsilon)\subset \prod_{i\in \llbracket 1~; ~ N\rrbracket} \mathrm{b}(x^i,\epsilon),\quad x=(x^1,x^2,\ldots,x^N)\in (\RR^d)^N.
$$
\begin{theo}\label{1.1.1}
	Let \( N \geq 2 \) and \( \chi > 0 \). Assume that \(\kk\) is a radial kernel. 
	\begin{enumerate}
		\item If \( \kk \) is an increasing function, then \( \mu \) is a Radon measure on \( E_\Ck \).
		\item If \( \kk \) is a decreasing function, then \( \mu \) is a Radon measure on \( (\RR^d)^N \).
	\end{enumerate}
\end{theo}
\begin{proof}
To prove part $(1)$, we fix a point $x \in E_\Ck$ and consider the partition $\mathrm{K}_1, \ldots, \mathrm{K}_\ell$ of $\llbracket 1~;~ N \rrbracket$ such that for all $p \neq q$ in $\llbracket 1~;~ \ell\rrbracket$, we have
$$
\begin{cases}
    x^i = x^j, & \text{for } i, j \in \mathrm{K}_p, \\
    x^i \neq x^j, & \text{for } i \in \mathrm{K}_p \text{ and } j \in \mathrm{K}_q,
\end{cases}
$$
We set
\[ \mathcal{O}_x = \prod_{i\in \llbracket 1~; ~ N\rrbracket}\mathrm{b}(x^i,r_x), \]
where
\[ r_x = \min \left(1, \min \left(\frac{\|x^i - x^j\|}{3} : i, j \in \llbracket 1~,~N\rrbracket\,\ \text{such that} \ x^i \neq x^j\right)\right) > 0 .\]
We begin by observing that
 $$\llbracket 1~;~N\rrbracket^2\setminus \mathrm{d}=\left(\bigcup_{p=1}^{\ell} \mathrm{K}_p\times \mathrm{K}_p\setminus \mathrm{d}\right)\bigcup \left(\bigcup_{\substack{p\not=q \\ p,q=1}}^{\ell} \mathrm{K}_p\times \mathrm{K}_q\right),\,\text{where}\,\, \mathrm{d}=\{(i,i):\,\, i\in \llbracket 1~;~N\rrbracket\}.$$
 Then for all $y \in \mathcal{O}_x$ we have
 
	{\small\begin{align*}
		m(y)&=\left(\prod_{p=1}^{\ell} \prod_{(i,j)\in \mathrm{K}_p^2\setminus \mathrm{d}}\exp\left(-\frac{\chi}{N} \kk(\|y^i-y^j\|)\right)\right)\times \left(\prod_{\substack{p\not= q\\ p,q=1}}^{\ell}\prod_{(i,j)\in \mathrm{K}_p\times \mathrm{K}_q}\exp\left(-\frac{\chi}{N} \kk(\|y^i-y^j\|)\right)\right)\\
		&:= m_1(y) \times m_2(y). \\ 
	\end{align*}}
 In the case of every distinct pair $p \neq q$ in $\llbracket1~;~\ell\rrbracket$, for  $i \in \mathrm{K}_p$ and $j \in  \mathrm{K}_q$, we obtain
 \[
	||y^i - y^j|| \geq ||x^i -x^j|| - ||x^i - y^i|| - ||x^j - y^j|| \geq ||x^i -x^j|| - 2r_x \geq r_x.
	\]
	Since $\kk$ is an increasing function, we have $
	m_2(y)\leq \exp\left(-\frac{\chi c_l}{N}\kk(r_x)\right)$ with $c_l=\sum_{\substack{p\not=q,\\ p,q=1}}^{\ell} \Card(\mathrm{K}_p\times \mathrm{K}_q)$.  Combining these inequalities, we obtain
 \begin{align}\label{weightineq}
		m(y)\leq C_x \times m_1(y),\quad y \in \mathcal{O}_x,	
	\end{align}
	where $C_x=\exp\left(-\frac{\chi c_l}{N}\kk(r_x)\right)$.  On the other hand, using the following identification
 {\small$$
	y=(y^1,y^2,\ldots,y^N)\in\prod_{i\in \llbracket 1~; ~ N\rrbracket}\mathrm{b}(x^i,r_x)\Leftrightarrow y=(Y^1,Y^2,\ldots,Y^{\ell})\in \prod_{p=1}^{\ell}\left(\prod_{i\in \mathrm{K}_p}\mathrm{b}(x^i,r_x)\right),
	$$}
 we find that
$$m_1(y)=\prod_{p=1}^{\ell}\left(\prod_{(i,j)\in \mathrm{K}_p^2\setminus \mathrm{d}}\exp\left(-\frac{\chi}{N} \kk(\|y^i-y^j\|)\right)\right)=\prod_{p=1}^{\ell}m_{1,p}(Y^p),$$
	where 
	$$m_{1,p}(Y^p)=\prod_{(i,j)\in \mathrm{K}_p^2\setminus \mathrm{d}}\exp\left(-\frac{\chi}{N} \kk(\|y^i-y^j\|)\right).$$
 Without loss of generality, we identify  $\mathrm{K}_p$ by the set $\llbracket 1~;~ n\rrbracket$ for any $p\in \llbracket 1~; ~\ell\rrbracket$. By the inequality \eqref{weightineq} we get 
	{\small \begin{align*}
		\mu(\mathcal{O}_x)
		&\leq C_x\prod_{p=1}^{\ell}\int_{(z^1,z^2,\ldots,z^n)\in (\mathrm{b}(0,1))^n}\prod_{(i,j)\in \llbracket 1~;~ n\rrbracket^2\setminus \mathrm{d}}\exp\left(-\frac{\chi}{N} \kk(\|z^i-z^j\|)\right) dz^1dz^2\ldots dz^n.
	\end{align*}}
 For all $i\in \llbracket1~;~n\rrbracket$, we set $z^i=(t^i_1,t^i_2,\ldots,t^i_d) \in \mathbb{R}^d$. Since $||z^i||\geq
	|t^i_k|$ with $k=1,\ldots, d$.
	We obtain that $I_n\leq J_n^d$, where
	$$I_n:= \int_{(z^1,z^2,\ldots,z^n)\in (\mathrm{b}(0,1))^n}\prod_{(i,j)\in \llbracket 1~;~ n\rrbracket^2\setminus \mathrm{d}}\exp\left(-\frac{\chi}{N} \kk(\|z^i-z^j\|)\right) dz^1dz^2\ldots dz^n,$$
	and
	$$
	J_n:=\int_{(t^1,t^2,\ldots,t^n)\in [-1,1]^n}\prod_{(i,j)\in \llbracket 1~;~ n\rrbracket^2\setminus \mathrm{d}}\exp\left(-\frac{\chi}{d N}\kk(|t^i-t^j|)\right)dt^1dt^2\ldots dt^n.
	$$ 
	Indeed, all of the integrals listed below are equal to each other. Thus by the inequality of arithmetic and geometric means, we obtain the following result
 \begin{align*}
		J_n
		&\leq \frac{1}{n}\sum_{i=1}^n\int_{(t^1,t^2,\ldots,t^n)\in [-1,1]^n} \prod_{(i,j)\in \llbracket 1~;~ n\rrbracket^2\setminus \mathrm{d}} \exp\left(-\frac{n\chi}{d N}\kk(|t^i-t^j|)\right)dt^1dt^2\ldots dt^n\\
		&=\int_{t\in [-1,1]}\left(\int_{s\in [-1,1]}\exp\left(-\frac{n\chi}{d N}\kk(|s-t|)\right)ds\right)^{n-1}dt	\\
		&\lesssim \left(\int_0\exp\left(-\frac{n\chi}{d N}\kk(r)\right)dr\right)^{n-1}.
	\end{align*}
	Since $x\in E_{\Ck}$, we get $n\leq \Ck$, hence $\mu(\mathcal{O}_x)<\infty$. \\
	To prove part $(2)$, let us consider $x \in (\mathbb{R}^d)^N$ and $y$ 
 $\in \prod_{i \in \llbracket 1~;~ N \rrbracket} \mathrm{b}(x^i, r_x)$. We note that
	\begin{align*}
		\|y^i-y^j\|\leq \|y^i-x^i\|+\|x^i-x^j\|+\|x^j-y^j\| \leq 2r_x+\|x^i-x^j\|.
	\end{align*}
	Since $-\kk$ is an increasing function, it follows that $
	-\kk(\|y^i-y^j\|)\leq -\kk(2r_x+\|x^i-x^j\|)$, for $(i,j)\in \llbracket 1~;~ N\rrbracket^2\setminus \mathrm{d}$. Consequently we have $$m(y)\leq \prod_{(i,j)\in \llbracket 1~;~N\rrbracket^2\setminus \mathrm{d}}\exp\left(-\kk(2r_x+\|x^i-x^j\|)\right).$$ Thus $\mu(\mathcal{O}_x)\lesssim C_x$, where $C_x$ depends only on $x$.
\end{proof}

For increasing functions \( \kk \) that do not exhibit a singularity near zero, i.e., \( \kk(0^+) \) exists, the integral
$$
\int_0 \exp\left(-\frac{n \chi}{d N} \kk(r)\right) dr 
$$
is finite for all $n\in \NN$. In particular,
$$
\Ck = \max\left\{ n \in \llbracket 2~;~ N+1 \rrbracket \,\,:\,\, \int_0\exp\left(-\frac{n \chi}{d N} \kk(r)\right) dr < \infty \right\} = N+1.
$$
Thus \( E_\Ck = (\RR^d)^N \). Otherwise, we obtain the following proposition.\\
\begin{propo}
	Let $N\geq 2$, and $\chi>0$. Assume that $\kk$ is an increasing radial function. If 
	$$
	\int_0\exp\left(-\frac{\Ck(\Ck-1)\chi}{N}\kk(r)\right)r^{d(\Ck-2)}dr=\infty,
	$$
	then  $\mu$ is not a Radon measure on $E_{\Ck+1}$.
\end{propo}
\begin{proof}
    It suffices to show that a compact set $ K $ in $ E_{\Ck+1}$ has infinite measure under $ \mu $. We set
    \[
	K = \prod_{i=1}^{\Ck} \Bar{\mathrm{b}}(0, 1) \times \prod_{k=\Ck+1}^{N} \Bar{\mathrm{b}}\left((2k,0,\ldots, 0), \frac{1}{2}\right)
	\]
	which is clearly a compact set in $E_{\Ck+1}$. We  observe that
{\small\[
	A = \{y = (y^1, \ldots, y^{\Ck}) : y^1, y^2 \in \mathrm{b}(0, 1/3), \forall i \in \llbracket 3~;~\Ck\rrbracket,\, y^i \in \mathrm{b}(y^1, \|y^1 - y^2\|/2)\} \subseteq (\Bar{\mathrm{b}}(0, 1))^{\Ck}.
	\]}
	For \(y \in A\), we have
	$||y^i - y^j|| \leq ||y^i - y^1|| + ||y^j - y^1|| \leq ||y^1 - y^2||$
	for all \(i, j = 1, \ldots, \Ck\), from which
 \[
	\prod_{(i,j)\in \llbracket 1~;~\Ck\rrbracket^2\setminus \mathrm{d}} \exp\left(-\frac{\chi}{N} \kk(||y^i-y^j||\right) \geq  \exp\left(-\frac{\Ck(\Ck-1)\chi}{N} \kk(||y^1-y^2||)\right).
\]
Consequently
	{\small \begin{align*}
		\mu(K) &\gtrsim  \int_{(\mathrm{b}(0,1))^{\Ck}} \prod_{(i,j)\in \llbracket 1~;~\Ck\rrbracket^2\setminus \mathrm{d}} \exp\left(-\frac{\chi}{N} \kk(||y^i-y^j||\right) \,dy^1\cdots dy^{\Ck}\\
		&\gtrsim \int_{(\mathrm{b}(0,1/3))^{2}} \exp\left(-\frac{\chi}{N} \kk(||y^1-y^2||)\right)^{\Ck(\Ck-1)}\times \\ &\qquad\qquad\qquad\qquad\qquad\qquad\left(\int_{(\mathrm{b}(y^1, \|y^1 - y^2\|/2))^{\Ck-2}} \,dy^3 \ldots \,dy^{\Ck}\right) dy^1dy^2\\
		&\asymp \int_{(\mathrm{b}(0,1/3))^{2}} \exp\left(-\frac{ \Ck(\Ck-1) \chi}{N} \kk(||y^1-y^2||)\right) \|y^1 - y^2\|^{d(\Ck-2)} dy^1dy^2\\
		& \asymp \int_{0}^{1/3} \exp\left(-\frac{\Ck(\Ck-1)\chi}{N} \kk( r)\right) r^{d(\Ck-2)} dr.
	\end{align*}}
	The proof is complete.
\end{proof}
\begin{exam}Let $N\geq 2$ and $\chi>0$.
\begin{enumerate}
		\item In the case of $d=2$ and  $\kk= \log(\cdot)$, \cite{Fournier2021CollisionsOT} showed that  $\Ck=\left\lfloor \frac{2N}{\chi} \right\rfloor$. Indeed, we have 
		\begin{align*}
			\Ck&=\max\left\{n\in \llbracket 2~;~N+1\rrbracket: \int_0 r^{-\frac{n \chi}{2N}}dr<\infty\right\}\\
			&=\left\lfloor \frac{2N}{\chi} \right\rfloor.
		\end{align*}
		\item In the case of $d\geq 3$ and \(\kk(x):=\kk_\alpha (x)= \frac{\|x\|^{1-\alpha}}{1-\alpha}\), where \(\alpha\in (0,1)\), we obtain 
		\begin{align*}
			\Ck&=\max\left\{n\in \llbracket2~;~ N+1\rrbracket:\,\, \int_0\exp\left(-\frac{n\chi}{ d N(1-\alpha)}r^{1-\alpha}\right)dr<\infty\right\}\\
			&= N+1,
		\end{align*}
		which implies that \(E_{\Ck}= (\mathbb{R}^d)^N\). \\
	\end{enumerate}
\end{exam}
Let \(N \geq 2\) and \(\chi > 0\). We define the nodal set \(\mathcal{N}\) as follows
$$
\mathcal{N} := \{ x \in (\mathbb{R}^d)^N : m(x) = 0 \}.
$$
The condition \(m(x) = 0\) holds if and only if there exist indices \((i, j) \in \llbracket 1 ~; ~ N \rrbracket^2 \setminus \mathrm{d}\) such that
$$
\exp \left( -\frac{\chi}{N} \kk(\| x^i - x^j \|) \right) = 0,
$$
which implies \(\kk(\| x^i - x^j \|) = +\infty\). If \(\kk\) is a regular increasing function such that \(|\kk(0^+)| = +\infty\), then there is no distance \(\| x^i - x^j \|\) that can make \(\kk(\| x^i - x^j \|) = +\infty\). Consequently the nodal set \(\mathcal{N}\) is empty. 
Conversely, if \(\kk\) is decreasing, \(\kk(\| x^i - x^j \|) = +\infty\) requires \(\| x^i - x^j \| = 0\), meaning \(x^i = x^j\) for some \(i \neq j\). Therefore the nodal set is
$$
\mathcal{N} = \{ x \in (\mathbb{R}^d)^N : \exists i \neq j \text{ such that } x^i = x^j \}.
$$
\begin{propo}\label{1.2.1}
	Let $N\geq 2$, and $\chi>0$. Assume that $\kk$ is a radial kernel. If $\kk$ is an increasing (resp., decreasing) function, then for any \( u \in C_0^\infty(E_{\Ck}) \) (resp., for any $u\in C_0^\infty(E_2)$), the self-adjoint operator 
	\begin{equation}\label{4.2}
		\Do u=\frac{1}{2}\Delta u -\frac{\chi}{N}\sum_{(i,k)\in \llbracket 1~;~N\rrbracket^2\setminus \mathrm{d}}\nabla_{x^k} \kk(x^i-x^k)\cdot\nabla_{x^k}u \end{equation}
	is $\mu$-symmetric.
\end{propo}
\begin{proof}
    For any $x = (x^1, x^2, \ldots, x^N) \in (\mathbb{R}^d)^N \setminus \mathcal{N}$, the gradient operator is defined as
\[\nabla = \begin{bmatrix}
		\nabla_{x^1} \\
		\nabla_{x^2}  \\
		\vdots \\
		\nabla_{x^N} 
	\end{bmatrix},
\]
where \( \nabla_{x^k} = (\partial_{x^k_1}, \partial_{x^k_2},\ldots,\partial_{x^k_d}) \) represents the gradient operator for the \( k \)-th  component. Leveraging the linearity of the derivative and the symmetry of the kernel \(\kk\), the calculation of \(\partial_{x^k_1} m(x)\) simplifies to
$$ \partial_{x^k_1} m(x)= -\frac{2\chi}{N} \sum_{i=1, i\neq k}^N \partial_{x^k_1}\kk(\|x^i-x^k\|) m(x).$$
Consequently
	$$
	\frac{\nabla m}{m}(x) = 
	-\frac{2 \chi}{N} \left(\sum_{i=1, i\neq k}^N \nabla_{x^k}\kk(\|x^i-x^k\|)\right)_{k=1}^N,
	$$	
and 

	\[
	\frac{\nabla m}{m}(x).\nabla = 
	-\frac{2\chi}{N} \sum_{(i,k)\in \llbracket 1~;~N\rrbracket^2\setminus \mathrm{d}}\nabla_{x^k} \kk(\|x^i-x^k\|)\cdot\nabla_{x^k}.
	\] 
	Therefore we get 
	$$
	\Do =  \frac{1}{2} \Delta +\frac{1}{2} \frac{\nabla 
		m}{m}(x) \cdot \nabla =\frac{1}{2m(x)} \operatorname{div}\left(m(x)\nabla  \right).
	$$
 Integration by parts shows that
	$$\int_{(\RR^d)^N} u(x)\Do v(x) \mu (dx)=\int_{(\RR^d)^N} v(x)\Do u(x) \mu (dx)
	$$
	for all $u,v \in C_0^\infty(E_{\Ck})$ (resp., $u,v\in C_0^\infty(E_2)).$
\end{proof}
\begin{propo}\label{PropClosable}
	Let $N\geq 2$, and $\chi>0$. Assume that $\kk$ is a radial kernel. If $\kk$ is an increasing (resp., decreasing) function, then the symmetric form \begin{equation}\label{8}
		\mathcal{E}_\kk(u , v) = \frac{1}{2} \int_{(\mathbb{R}^d)^N} \nabla u.\nabla v  \mu(dx),\quad u,v\in C^{\infty}_0(E_{\Ck})\,\,(\text{resp.,}\,\, u,v\in C^{\infty}_0(E_{2} ))
	\end{equation}
	is closable. 
\end{propo}	
\begin{proof}
According to the section \ref{sec.3}, it suffices to show that $$S(m):=\left\{x\in E_{\Ck}\,\,:\,\, \int_{y\in \mathrm{B}(x,\epsilon)}\frac{1}{m(y)}dy=\infty\,\, \text{for all}\,\,  \epsilon>0\right\}$$ is empty. Indeed, 
{\small\begin{align*}
		\int_{y\in \mathrm{B}(x,\epsilon)}\frac{1}{m(y)}dy&\leq\int_{(y^1,y^2,\ldots,y^N)\in \underset{i\in \llbracket 1~;~N\rrbracket}{\prod\mathrm{b}(x^i,\epsilon)}}\left(\prod_{p=1}^{\ell} \prod_{(i,j)\in \mathrm{K}_p^2\setminus \mathrm{d}}\exp\left(\frac{\chi}{N} \kk(\|y^i-y^j\|)\right)\right) \times\\
		&\quad \qquad  \qquad  \left(\prod_{\substack{p\not= q\\ p,q=1}}^{\ell}\prod_{(i,j)\in \mathrm{K}_p\times \mathrm{K}_q}\exp\left(\frac{\chi}{N} \kk(\|y^i-y^j\|)\right)\right)dy^1\ldots dy^N\\
		& :=  \int_{(y^1,y^2,\ldots,y^N)\in \prod_{i\in \llbracket 1~;~N\rrbracket}\mathrm{b}(x^i,\epsilon)} n_1(y)\times n_2(y) dy^1\ldots dy^N.
	\end{align*}}
 For $(y^1,y^2,\ldots,y^N)\in \prod_{i\in \llbracket 1~;~N\rrbracket}\mathrm{b}(x^i,\epsilon)$, we obtain that
	$
	||y^i - y^j|| \leq ||y^i - x^i|| + ||x^j - y^j|| \leq 2\epsilon
	$
	for all \((i, j) \in \mathrm{K}_p^2 \setminus \mathrm{d}\). On the other hand, $
	||y^i - y^j|| \leq 2\epsilon + ||x^i - x^j||$ for all $(i, j) \in \mathrm{K}_p \times \mathrm{K}_q$ with $p \neq q$.
	Since $\kk$ is increasing, then  
	$$
	n_1(y)\leq \exp\left(\frac{d_l\chi}{N}\kk(2\epsilon)\right), \quad d_l=\sum_{p=1}^l\sum_{(i,j)\in \mathrm{K}_p^2\setminus \mathrm{d}}1,
	$$
	and
	$$
	n_2(y)\leq \prod_{\substack{p\not= q\\ p,q=1}}^{\ell}\prod_{(i,j)\in \mathrm{K}_p\times \mathrm{K}_q}\exp\left(\frac{\chi}{N} \kk(2\epsilon+\|x^i-x^j\|)\right)
	\leq \exp\left(\frac{e_l\chi}{N}\kk(2\epsilon + M_x)\right),
	$$
	where $M_x=\max\left\{\|x^i-x^j\|\,\,:\,\, p\not=q,\,\, p,q=1,\ldots,\ell,\,\, (i,j)\in \mathrm{K}_p\times \mathrm{K}_q\right\}$ and $e_l=\sum_{\substack{p\not= q\\ p,q=1}}^{\ell}\sum_{(i,j)\in \mathrm{K}_p\times \mathrm{K}_q}1$. Consequently 
	$$
	\int_{y\in \mathrm{B}(x,\epsilon)}\frac{1}{m(y)}dy\leq \mathrm{C}(x,\epsilon),
	$$
 where $\mathrm{C}(x, \epsilon)$ depends only on $x$ and $\epsilon$. We conclude that $S(m) = \emptyset$. Consequently $(\mathcal{E}_\kk, C^\infty_0(E_\Ck))$ is closable. In the case where $\kk$ is decreasing, a similar method shows that the set $S(m)$ defined by 
\[
S(m) = \{ x \in E_2  : \int_{y \in \mathrm{B}(x, \epsilon)} \frac{1}{m(y)} \, dy = \infty \   \text{for all} \ \epsilon>0\}
\]
is empty. Then we obtain that $(\mathcal{E}_\kk, C^\infty_0(E_2))$ is closable.

\end{proof}
\begin{proof}[Proof of Theorem~{\upshape\ref{1.1}}] Without loss of generality, we assume that  $\kk$ is an increasing function. Consequently from Theorem \ref{1.1.1} and Proposition \ref{1.2.1}, we extend the Dirichlet form $\mathcal{E}_\kk$ to the domain $\mathcal{D(E_\kk)}={\overline{C_0^{\infty}(E_{\Ck})}}^{\mathcal{E}_1}$, where $$\mathcal{E}_1(u,u)=\int_{\left(\mathbb{R}^d\right)^N}\left (u^2+\frac{1}{2}\|\nabla u\|^2\right)  \mu(dx).$$ 
It is apparent that $\left(\mathcal{E_\kk}, \mathcal{D(E_\kk)}\right)$ is a regular form with a core $C_0^{\infty}(E_{\Ck})$. It is also strongly local, as $\mathcal{E}(u, v) = 0$ for $u, v \in C_0^{\infty}(E_{\Ck})$ whenever $v$ constant on a neighborhood of $\supp[u]$.  Therefore Section \ref{Dirichletform} implies the existence of a  $\mu$-symmetric diffusion process $$\mathbb{X} = ((\mathcal{X}_t)_{t \geq 0}, (\mathbb{P}_{x \in \mathbb{E}_{\Delta}}))$$ 
with state space $\mathbb{E}_{\Delta}$ such that for quasi-everywhere $x\in E_{\Ck}$ and for all $\varphi \in \mathcal{D}(\Do)$
$$\varphi(\mathcal{X}_t)-\varphi(x)-\int_0^t \Do\varphi(\mathcal{X}_s)ds$$ 
is a $\mathbb{P}_x$-martingale. Consequently $(\mathcal{X}_t)_{t \geq 0}$ weakly solves (\ref{3})  for $x\in E_{\Ck}$ outside a set of capacity zero. However the solution is not necessarily unique, as it starts from a point that avoids polar sets, whose explicit characterization can be challenging. To address this issue, we employ the \(\mathrm{H_2}\)-condition as introduced by \cite{albeverio2003strong}, which is presented in the following proposition.
\begin{propo}\label{4.4}
	Let $p>1$, and $\delta>0$. Assume that $\kk$ is a radial kernel, we have 
	$$
	\int_{x\in \mathrm{B}(0,\delta)}\left\|\frac{\nabla m}{m}\right\|^pm(x)dx\lesssim \left(\int_0 |\kk'(t)|^p \exp\left(-\frac{\chi}{N}\kk(t)\right) t^{d-1}dt
	\right)^{p/2}.$$
\end{propo}
\begin{proof}
    Recall that for $x\in (\RR^d)^N \setminus \mathcal{N}$, we have
	$$\frac{\nabla m}{m}(x) = 
	-\frac{2 \chi}{N} \left(\sum_{i=1,
		i\neq k}^N\nabla_{x^k}\kk(\|x^i-x^k\|)\right)_{k=1}^N.$$
Consequently
  $$
	\left\|\frac{\nabla m}{m}\right\|^p=\left(\frac{2\chi}{N}\right)^p\left(\sum \partial x^{k}_l\kk(\|x^i-x^k\|)\partial x^k_l\kk(\|x^j-x^k\|)\right)^{p/2},
	$$
 where the sum over $l\in \llbracket 1~;~N\rrbracket$, $k\in \llbracket 1~;~N\rrbracket$ and $i,j\not=k \in \llbracket 1~;~N\rrbracket$. Then
 {\small\begin{align*}
		\int_{x \in \mathrm{B}(0, \delta)} \left\|\frac{\nabla m}{m}\right\|^p m(x) \, dx
		&\leq \left(\frac{2\chi}{N}\right)^p \left(\sum \left\|\partial x^k_l\kk(\|x^i-x^k\|)\partial \kk(\|x^j-x^k\|)\right\|_{L^{p/2}(\mathrm{B}(0,\delta))}\right)^{p/2}.
	\end{align*}}
 We denote by $I_{i,j}^{k,l}=\|\partial x^k_l\kk(\|x^i-x^k\|)\partial x^k_l\kk(\|x^j-x^k\|)\|^{p/2}_{L^{p/2}(\mathrm{B}(0,\delta))}$, therefore
{\small \begin{align*}
I_{i,j}^{k,l}&=\int_{x\in \mathrm{B}(0,\delta)}\left|\partial x^k_l \kk(\|x^i-x^k\|)\right|^{p/2}\left|\partial x^k_l\kk(\|x^j-x^k\|)\right|^{p/2}\times\\ 
  &\qquad\qquad\qquad\qquad\qquad\qquad\qquad\qquad\exp\left(-\frac{\chi}{N}\sum_{(i_0,j_0) \in \llbracket 1~;~N\rrbracket^2\setminus \mathrm{d}} \kk(\|x^{i_0}-x^{j_0}\|)\right) dx\\
		&\leq \int_{x\in (\mathrm{b}(0,\delta))^N}\left|\partial x^k_l \kk(\|x^i-x^k\|)\right|^{p/2}\left|\partial x^k_l\kk(\|x^j-x^k\|)\right|^{p/2}\times\\ 
  &\qquad\qquad\qquad\qquad\qquad\qquad\qquad\qquad\exp\left(-\frac{\chi}{N}\sum_{(i_0,j_0) \in \llbracket 1~;~N\rrbracket^2\setminus \mathrm{d}} \kk(\|x^{i_0}-x^{j_0}\|)\right) dx\\
		&\lesssim \int_{x^k\in \mathrm{b}(0,\delta)}\int_{x^i\in \mathrm{b}(0,\delta)}\left|\partial x^k_l \kk(x^i-x^k)\right|^{p} \exp\left(-\frac{\chi}{N}\|x^i-x^k\|\right)dx^idx^k\\
		&\asymp \int_{x^k\in \mathrm{b}(0,\delta)}\int_{x^i\in \mathrm{b}(0,\delta)}\left(\frac{|x^i_l-x^k_l|}{\|x^i-x^k\|}\right)^p|\kk'(\|x^i-x^k\|)|^p \exp\left(-\frac{\chi}{N}\kk(\|x^i-x^k\|)\right)dx^i dx^k\\
		&\lesssim  \int_{x^k\in \mathrm{b}(0,\delta)}\int_{x^i\in \mathrm{b}(0,\delta)}|\kk'(\|x^i-x^k\|)|^p \exp\left(-\frac{\chi}{N}\kk(\|x^i-x^k\|)\right)dx^i dx^k\\
		&\lesssim \int_{y\in \mathrm{b}(0,2\delta)}|\kk'(\|y\|)|^p \exp\left(-\frac{\chi}{N}\kk(\|y\|)\right) dy\\
		&\asymp \omega_{d-1}\int_0^{2\delta}|\kk'(t)|^p \exp\left(-\frac{\chi}{N}\kk(t)\right)t^{d-1}dt,
	\end{align*}}
where $\omega_{d-1}$ denotes the $(d-1)$-dimensional measure of the sphere $S^{d-1}(0,2\delta)$. The proof is complete.
\end{proof}
Based on the condition $[\mathrm{H}]_p$ that holds for some $p>d N$ and Proposition \ref{4.4}, we deduce that the  $\mathrm{H}_2$-condition is satisfied, allowing us to obtain the result.  
\end{proof}
\section{Generalized Schrödinger Operator}
Let $\Do$ be the Dirichlet operator (DO) associated with the smallest closed extension of $\mathcal{E}_{\kk}$ (see Proposition \ref{1.2.1}).  We set $\phi(x) = \sqrt{m(x)}$ and  we consider the generalized Schrödinger operator (GSO), which is formally expressed as follows
$$\Ha := \phi \Do \phi^{-1}=-\Delta+V,\quad \text{where}\,\, V=\frac{\Delta \phi}{\phi}.$$
In light of Theorem $2.6$ from \cite{Albeverio1977EnergyFH}, we deduce the explicit expression of the GSO given below.\\

\begin{theo}\label{Schro}
	Let $N\geq 2$, and $\chi>0$. Assume that $\kk$ is a radial kernel, if 
	\begin{equation}
		\int_{0}\left[ \left( \kk''(t) + (d-1) \frac{\kk'(t)}{t} \right)^2+(\kk'(t))^4\right]\exp \left( -\frac{\chi}{N} \kk(t) \right) t^{d-1} \, dt<\infty,
	\end{equation}
	then the GSO is given explicitly as follows 
	\begin{align*}
	    \Ha &= -\Delta-\frac{\chi}{2N} \sum_{(i,j)\in \llbracket 1~;~N\rrbracket^2\setminus \mathrm{d}} \Delta\kk(\|x^i-x^j\|) +\\ 
&\qquad\qquad\qquad\qquad\qquad\qquad\qquad\left(\frac{\chi}{2N}\right)^2 \left( \sum_{k=1}^N\sum_{l=1}^d \left(\sum_{(i,j)\in \llbracket 1~;~N\rrbracket^2\setminus \mathrm{d}}  \partial_{x^k_l}\kk(\|x^i-x^j\|)\right)^2 \right).
	\end{align*}
\end{theo}
The proof of Theorem \ref{Schro} is based on the following lemma.\\

\begin{lemma}\label{Lm}
	Let $N\geq 2$, and let $\chi>0$, and $\delta>0$. Assume that $\kk$ is a radial kernel, we have 
	$$
	\int_{x\in \mathrm{B}(0,\delta)}|\Delta \phi(x)|^2dx\lesssim \int_{0}^{2\delta}\left[ \left( \kk''(t) + (d-1) \frac{\kk'(t)}{t} \right)^2+(\kk'(t))^4\right]\exp \left( -\frac{\chi}{N} \kk(t) \right) t^{d-1} \, dt.
	$$ 
\end{lemma}
\begin{proof}
For any $x\in (\RR^d)^N$, we define 
	$$\phi(x):= -\frac{\chi}{2N} \sum_{(i,j)\in \llbracket 1~;~N\rrbracket^2\setminus \mathrm{d}} \partial_{x^k_l}\kk(\|x^i-x^j\|).$$
	The second partial derivative of \(\phi(x)\) with respect to \(x^k_l\)  is computed as follows
	{\small$$\partial^2_{x^k_l}\phi(x)=-\frac{\chi}{2N}\left(\sum_{(i,j)\in \llbracket 1~;~N\rrbracket^2\setminus \mathrm{d}} \partial^2_{x^k_l}\kk(\|x^i-x^j\|) -\frac{\chi}{2N}\left(\sum_{(i,j)\in \llbracket 1~;~N\rrbracket^2\setminus \mathrm{d}}\partial_{x^k_l}\kk(\|x^i-x^j\|)\right)^2\right)\phi(x).$$}
Each component \(\Delta_{x^k} \phi(x)\) is the sum of the second partial derivatives with respect to the coordinates of \(x^k\). Specifically 
	\begin{align*}
	\Delta_{x^k}\phi(x) &= -\frac{\chi}{2N} \sum_{(i,j)\in \llbracket 1~;~N\rrbracket^2\setminus \mathrm{d}}\Delta_{x^k}\kk(\|x^i-x^j\|)\phi(x) +\\ &\qquad\qquad\qquad\qquad\qquad\qquad\left(\frac{\chi}{2N}\right)^2 \left( \sum_{l=1}^\mathrm{d} \left(\sum_{(i,j)\in \llbracket 1~;~N\rrbracket^2\setminus \mathrm{d}}\partial_{x^k_l}\kk(\|x^i-x^j\|)\right)^2 \right)\phi(x).
	\end{align*}
	Consequently we find that
	\begin{align*}
	    	\Delta\phi(x) &= -\frac{\chi}{2N} \sum_{(i,j)\in \llbracket 1~;~N\rrbracket^2\setminus \mathrm{d}} \Delta\kk(\|x^i-x^j\|)\phi(x) + \\ 
&\qquad\qquad\qquad\qquad\left(\frac{\chi}{2N}\right)^2 \left( \sum_{k=1}^N\sum_{l=1}^d \left(\sum_{(i,j)\in \llbracket 1~;~N\rrbracket^2\setminus \mathrm{d}}  \partial_{x^k_l}\kk(\|x^i-x^j\|)\right)^2 \right)\phi(x).
		\end{align*}

 We set
	$$I_1(x):= -\frac{\chi}{2N} \sum_{(i,j)\in \llbracket 1~;~N\rrbracket^2\setminus \mathrm{d}} \Delta\kk(x^i,x^j)\phi(x),$$
	and
	$$I_2(x):= \left(\frac{\chi}{2N}\right)^2 \left( \sum_{k=1}^N\sum_{l=1}^d \left(\sum_{(i,j)\in \llbracket 1~;~N\rrbracket^2\setminus \mathrm{d}}  \partial_{x^k_l}\kk(x^i,x^j)\right)^2 \right)\phi(x).$$
 We fix  $\delta>0$, then 
	$$\int_{\mathrm{B}(0,\delta)}|\Delta \phi(x) |^2 dx \lesssim \int_{\mathrm{B}(0,\delta)} |I_1(x)|^2dx+ \int_{\mathrm{B}(0,\delta)} |I_2(x)|^2dx:={\mathcal{I}_1}+{\mathcal{I}_2}.$$
	By the Cauchy-Schwarz inequality, we have
	$$\mathcal{I}_1 \leq \left(\frac{\chi}{2}\right)^2\sum_{(i,j)\in \llbracket 1~;~N\rrbracket^2\setminus \mathrm{d}} \int_{\mathrm{B}(0,\delta)}\left[\Delta\kk(\|x^i-x^j\|)\right]^2 m(x) dx, $$
	where $\Delta_{x^k}\kk(\|x^i-x^j\|)=\kk''(\|x^j-x^i\|)+(d-1)\frac{\kk'(\|x^j-x^i\|)}{\|x^j-x^i\|}$. Consequently we obtain
 \begin{align*}
		\mathcal{I}_1 
		&\leq \left(\frac{\chi}{2}\right)^2 \sum_{(i,j)\in \llbracket 1~;~N\rrbracket^2 \setminus \mathrm{d}}  
		\left( \int_{\tilde{x} = (x^q)_{q \not= i, j} \in (\mathrm{b}(0,\delta))^{N-2}} \tilde{m}(\tilde{x}) \, \mathrm{d}\tilde{x} \right) \times  \\
		& \int_{x^j \in \mathrm{b}(0,\delta)} \int_{x^i \in \mathrm{b}(0,\delta)} \left( \kk''(\|x^j - x^i\|) + (d-1) \frac{\kk'(\|x^j - x^i\|)}{\|x^j - x^i\|} \right)^2 \times\\ 
  &\quad\qquad\qquad\qquad\qquad\qquad\qquad\qquad\qquad\exp \left( -\frac{\chi}{N} \kk(\|x^j - x^i\|) \right) \, dx^i \, dx^j \\
		&\lesssim  \left(\frac{\chi}{2}\right)^2  \int_{y \in \mathrm{b}(0, 2\delta)} \left( \kk''(\|y\|) + (d-1) \frac{\kk'(\|y\|)}{\|y\|} \right)^2 \exp \left( -\frac{\chi}{N} \kk(\|y\|) \right) \, dy \\
		&\asymp\left(\frac{\chi}{2}\right)^2\times\omega_{d-1} \int_{0}^{2\delta} \left( \kk''(t) + (d-1) \frac{\kk'(t)}{t} \right)^2 \exp \left( -\frac{\chi}{N} \kk(t) \right) t^{d-1} \, dt.
	\end{align*}
	On the other hand, we have
	\begin{align*}
		\mathcal{I}_2&\leq \left(\frac{dN\chi^4}{16}\right) \sum_{(i,j)\in \llbracket 1~;~N\rrbracket^2\setminus \mathrm{d}}
		\left( \int_{\tilde{x} = (x^q)_{q \not= i, j} \in (\mathrm{b}(0,\delta))^{N-2}} \tilde{m}(\tilde{x}) \, d\tilde{x} \right) \times \\
		&\qquad \qquad \int_{x^j \in \mathrm{b}(0,\delta)} \int_{x^i \in \mathrm{b}(0,\delta)}\left(\kk'(\|x^j-x^i\|\right)^4\exp\left(-\frac{\chi}{N}\kk(\|x^j-x^i\|\right)dx^i dx^j\\
		& \lesssim \left(\frac{dN\chi^4}{16}\right) \times \omega_{d-1}\int_{0}^{2\delta} \left(\kk'(t)\right)^4\exp\left(-\frac{\chi}{N}\kk(t\right)t^{d-1} dt.
	\end{align*}
	Therefore 
	$$
	\int_{\mathrm{B}(0,\delta)}|\Delta \phi(x) |^2 dx \lesssim  \int_{0}^{2\delta} \left( \kk''(t) + (d-1) \frac{\kk'(t)}{t} +\left(\kk'(t)\right)^4\right)^2 \exp \left( -\frac{\chi}{N} \kk(t) \right) t^{d-1} \, dt.
	$$
	The proof is complete.
\end{proof}
\begin{proof}[Proof of Theorem \ref{Schro}]
	According to Lemma \ref{Lm}, we obtain that the Laplacian $\Delta \phi$ belongs to $L^2_{loc}(dx)$, and according to Theorem $2.6$  from \cite{Albeverio1977EnergyFH}, we obtain the result.
\end{proof}
Let us illustrate this general result with a few examples.\\

\begin{exam}
	\begin{enumerate}
 \item Let $N \geq 2$, $\chi > 0$, and $d > 4 + \chi/N$. We consider $\kk = \log(\cdot)$. Under this framework, the function $\phi(x)$ is defined by
 $$\phi(x)=\prod_{(i,j)\in \llbracket 1~;~N\rrbracket^2\setminus \mathrm{d} }\|x^i-x^j\|^{\frac{-\chi}{2N}}.$$
		Consequently the DO is given by
		$$\Do\Big|_{C^\infty_0}=-\Delta +\frac{2\chi}{N}\sum_{(i,k)\in \llbracket 1~;~N\rrbracket^2\setminus \mathrm{d}}\frac{x^i-x^k}{\|x^i-x^k\|^2}\cdot\nabla_{x^k}.$$
		Moreover it follows that $\Delta_{x^k} \log(\cdot)= 0$. Subsequently the GSO can be expressed as $$\Ha\Big|_{C^\infty_0}=-\Delta +
		\left(\frac{\chi}{N}\right)^2  \left( \sum_{l=1}^d \left(\sum_{(i,k)\in \llbracket 1~;~N\rrbracket^2\setminus \mathrm{d}} \frac{x^i_l-x^k_l}{\|x^i-x^k\|^2}\right)^2 \right).$$
		
		\item  Let $N \geq 2$, $\chi > 0$, and $d > 4 - 2\alpha$ for $\alpha \in (0,1)$. We consider $\kk_{\alpha}(\|x^i - x^j\|) = \frac{\|x^i - x^j\|^{1 - \alpha}}{1 - \alpha}$. The function $\phi(x)$ takes the form
        $$\phi(x)=\exp\left(\frac{-\chi}{2N(1-\alpha)}\sum_{(i,j)\in \llbracket 1~;~N\rrbracket^2\setminus \mathrm{d}}\|x^i-x^j\|^{1-\alpha}\right).$$
		This analysis culminates in the expression of the DO as follows $$\Do\Big|_{C^\infty_0}=-\Delta +\frac{2\chi}{N}\sum_{(i,k)\in \llbracket 1~;~N\rrbracket^2\setminus \mathrm{d}}\frac{x^i-x^k}{\|x^i-x^k\|^{1+\alpha}}\cdot\nabla_{x^k}.$$
		Moreover we find $\Delta_{x^k} \|x^i-x^k\|^{1-\alpha}= \frac{(1-\alpha)^2}{\|x^i-x^k\|^{1+\alpha}}$. As a result, the GSO can be represented as 
  \begin{align*}
		\Ha\Big|_{C^\infty_0} &=  -\Delta - \frac{\chi(1-\alpha)}{N} \sum_{(i,j)\in \llbracket 1~;~N\rrbracket^2\setminus \mathrm{d}} \frac{1}{\|x^i - x^j\|^{1+\alpha}} +\\ 
&\quad\qquad\qquad\qquad\qquad\qquad\left(\frac{\chi}{N}\right)^2  \left( \sum_{l=1}^d \left(\sum_{(i,k)\in \llbracket 1~;~N\rrbracket^2\setminus \mathrm{d}} \frac{x^i_l-x^k_l}{\|x^i-x^k\|^2}\right)^{1+\alpha} \right).
   \end{align*}
  
	\end{enumerate}
\end{exam}
\section{Discussion}
We model the energy states of individual particles using the Schrödinger operator. Theorem \ref{1.1} establishes that under general conditions on the kernel function $\kk$, a specific diffusion process denoted by $(\mathcal{X}_t)_{t\geq 0}$ governs the solution to the aggregation-diffusion equation
$$(\partial_t + \Do)\rho_t=0.$$
This equation describes the evolution of a particle density $\rho$. Interestingly, the solution to the Schrödinger equation
$$(\partial_t +\Ha)\psi_t=0,$$
can be expressed simply as the diffusion process $(\mathcal{X}_t)_{t\geq 0}$ weighted by a function $\phi$. Mathematically, this relationship translates to $\psi_t(x)=\mathcal{L}(\phi(\mathcal{X}_t))$. This observed link between the Dirichlet and Schrödinger operators suggests a potential duality between the wave-particle view in quantum mechanics and the behavior of biological systems modeled by aggregation-diffusion equations.\\

Can this connection provide deeper insights into the behavior of biological systems at the quantum level?\\

Certain bacteria exhibit a "run and tumble" motion, leading their trajectories to be more accurately described by Lévy flights rather than Brownian motion (see \cite{salem2019propagation}). This observation motivates the substitution of classical diffusion in the density evolution equation of bacteria with fractional diffusion. Let $N\geq 2$ and let $\chi>0$, let $(\mathcal{Z}^i_t)_{i=1,\ldots,N, t\geq 0}$ be $N$ independent $a$-stable Lévy flights on $\RR^d$. Thus we obtain the following equation
\begin{equation*}
	\begin{cases}
		dX_t^{i} = \sqrt{2}d\mathcal{Z}_t^{i} -\frac{\chi}{N}\sum_{j=1, i \neq j}^N \nabla\kk(X_t^{i}-X_t^{j})dt,\\
		X_0=x \sim \otimes_{i=1}^N \rho_0.
	\end{cases}
\end{equation*}
What are the limits of the method following the transition from local to non-local forms?

\bmhead{Acknowledgements}
The authors would like to express their sincere gratitude to Abdellah Alla and Mohamed Toumlilin for their insightful comments and valuable suggestions.

\bibliography{sn-bibliography}

\end{document}